\documentclass[10pt]{amsart}
\usepackage{amsmath,amssymb,xypic,array}
\usepackage[T1]{fontenc}
\usepackage{amsfonts}
\usepackage{amsthm}
\usepackage{setspace}
\usepackage{tikz}
\usepackage{booktabs}
\usepackage{hyperref}
\usepackage{longtable}
\usepackage{geometry}
\usetikzlibrary{trees}

\usepackage[english]{babel}
\usepackage[latin1]{inputenc}
\usepackage{amsmath,amssymb,amsthm,epsfig}
\usepackage{color}
\usepackage{mathrsfs}
\usepackage{latexsym}
\usepackage{textcomp}

\newtheorem{mthm}{Theorem}
\newtheorem{Lemma}[mthm]{Lemma}
\newtheorem{Obs}[mthm]{Remark}
\newtheorem{Prop}[mthm]{Proposition}
\newtheorem{Cor}[mthm]{Corollary}

\newcommand{\op}[1]{{\mathcal O}_{\mathbb{P}^{#1}}}
\newcommand{\p}[1]{{\mathbb{P}^{#1}}}
\newcommand{\tp}[1]{{\rm T}{\mathbb{P}^{#1}}}

\def\P{{\mathbb{P}}}

\def\Z{{\mathbb{Z}}}

\def\sF{{\mathscr{F}}}
\def\sD{{\mathscr{D}}}

\def\sG{{\mathscr{G}}}
\def\sH{{\mathscr{H}}}

\def\O{{\mathcal{O}_{X}}}
\def\OZ{{\mathcal{O}_{Z}}}
\def\IZ{I_{Z}}
\def\OC{{\mathcal{O}_{C}}}

\def\OX{{\mathcal{O}_{X}}}

\newcommand{\inext}{{\mathcal E}{\it xt}}
\DeclareMathOperator{\sing}{Sing}

\DeclareMathOperator{\Hom}{Hom}
\DeclareMathOperator{\Ext}{Ext}
\DeclareMathOperator{\rk}{{rk}}
\DeclareMathOperator{\coker}{{coker}}
\DeclareMathOperator{\img}{{im}}

\DeclareMathOperator{\Pic}{{Pic}}

\title[Distributions and rank 2 reflexive sheaves on threefolds]{Codimension one distributions and stable rank 2 reflexive sheaves on threefolds}

\author[ O. Calvo-Andrade ]{ O. Calvo-Andrade }
\thanks{ }
\dedicatory{}
\address{CIMAT \\ Ap. Postal 402, Guanajuato, 36000, Gto. Mexico}
\email{omegar.mat@gmail.com}

\author[M. Corr\^ea]{M. Corr\^ea}
\thanks{ }
\dedicatory{}
\address{ICEX-UFMG \\ Departamento de Matem\'atica \\
Av. Ant\^onio Carlos, 6627 \\ 31270-901 Belo Horizonte-MG, Brazil}
\email{mauriciojr@ufmg.br}

\author[  M. Jardim]{M. Jardim}
\thanks{ }
\dedicatory{}
\address{IMECC - UNICAMP \\ Departamento de Matem\'atica \\
Rua S\'ergio  Buarque de Holanda, 651\\ 13083-970 Campinas-SP, Brazil}
\email{jardim@ime.unicamp.br}

\date{}

\begin{document}

\maketitle

\begin{abstract}
We show that codimension one distributions with at most isolated singularities on certain smooth projective threefolds with  Picard rank one have stable tangent sheaves. The ideas in the proof of this fact are then applied to the characterization of certain irreducible components of the moduli space of stable rank 2 reflexive sheaves on $\p3$, and to the construction of stable rank 2 reflexive sheaves with prescribed Chern classes on general threefolds. We also prove that if $\sG$ is a  subfoliation of a codimension one distribution $\sF$ with isolated singularities, then $\sing(\sG)$ is a curve. As a consequence,   we give a criterion to decide whether $\sG$ is globally given as the intersection of $\sF$ with another codimension one distribution.
Turning our attention to codimension one distributions with non isolated singularities, we determine the number of connected components of the pure 1-dimensional component of the singular scheme.
\end{abstract}

\section{Introduction}

A \emph{codimension $r$ distribution} $\sF$ on a smooth complex manifold $X$ is given by an exact sequence
\begin{equation}\label{eq:Dist}
\mathscr{F}:\  0  \longrightarrow T_\sF \stackrel{\phi}{ \longrightarrow} TX \stackrel{\pi}{ \longrightarrow} N_{\sF}  \longrightarrow 0,
\end{equation}
where $T_\sF$ is a reflexive sheaf of rank $s:=\dim(X)-r$, and $N_{\sF}$ is a torsion free sheaf; these are respectively called the \emph{tangent} and the \emph{normal} sheaves of $\mathscr{F}$, respectively. We will use the notation $L_{\sF}:=\det(N_{\sF})$. 

Taking the maximal exterior power of the dual morphism $\phi^\vee:\Omega^1_X\to T_\sF^\vee$ we obtain a morphism $\Omega^{s}_X\to \det(T_\sF)^\vee$; its image is an ideal sheaf $I_{Z/X}$ of a subscheme $Z\subset X$, called the \emph{singular scheme} of $\mathscr{F}$, twisted by $\det(T_\sF)^\vee$.

Finally, we say that $\sF$ is integrable, that is a \emph{foliation}, if $ T_\sF$ is closed under the Lie bracket. Clearly, every distribution of codimension $\dim X-1$ is integrable, being given by a single vector field.  For more details about distributions and foliations  see (Araujo $\&$ Corr\^ea,  Corr\^ea et al. 2015a, Corr\^ea Jr et al. 2015, Esteves $\&$ Kleiman 2003).

This paper is dedicated to codimension one distributions on smooth projective threefolds. More precisely, we consider smooth projective threefolds $X$ with Picard group generated by an ample line bundle $\mathcal{O}_X(1)$. In addition, we assume that $H^1(\mathcal{O}_X(t))=0$, for all $t\in \mathbb{Z}$. We set
\begin{itemize}
\item $c_X:=c_1(TX)=-c_1(\Omega_X^1)$,
\item $\rho_X:=\min\{ t\in \Z ~|~  H^0( \Omega_X^1(t))\neq 0\}$. 
\end{itemize}
Examples of projective threefolds satisfying the desired conditions are smooth weighted projective complete intersection threefolds $X$; this claim follows from (Araujo et al. 2018, Lemma 5.17 andCorollary 5.23) and (Flenner 1981, Satz 8.11); furthermore, $\rho_X=2$ and $TX$ is always $\mu$-stable. More generally, projective manifolds with \emph{special cohomology} in the sense of (Peternell $\&$ Wisniewski 1995) also satisfy the desired conditions.



Our first two main results concern distributions with only isolated singularities on such threefolds.

\begin{mthm}\label{StGDist}
Let $\sF$ be a codimension one distribution on a smooth projective threefold $X$ with $\Pic(X)=\Z$ and $H^1(\mathcal{O}_X(t))=0$, for all $t\in \mathbb{Z}$. Assume that $\sing(\sF)$ is either empty or has dimension equal to zero. If $c_1(T_\sF)<(\leq)~2\rho_X$, then $T_\sF$ is $\mu$-(semi)-stable. Moreover, if $TX$ is $\mu$-(semi)stable, then $T_\sF$ is $\mu$-(semi)-stable.
\end{mthm}

Codimension one distributions $\sF$ with the property assumed in the previous statement, namely that $\sing(\sF)$ is either empty or has dimension equal to zero, are called \emph{generic}, because they can be defined by a general 1-form $\omega\in H^0(\Omega^1_X(d))$ for some $d$. If $X$ is a projective variety with cyclic Picard group, then it follows from (Brunella $\&$ Perrone  2011)  that an integrable codimension one distribution always has non isolated singular points. In particular, generic codimension one distributions on projective threefolds with rank one Picard group cannot be integrable. 

Moreover, we recall that if $\sing(\sF)$ is empty, then it is non-integrable and $(X,\sF)\simeq (\p3, \sD) $, where $\sD$ is the contact distribution on $\p3$, whose tangent bundle is the null correlation bundle twisted by $\mathcal{O}_{\p3}(1)$, see  (Ye 1984). 

Next, we consider subfoliations of generic codimension one distributions. We   prove that if $\sG$ is a subfoliation of a distribution $\sF$ with isolated singularities, then $\sG$ has non isolated singularities. As a consequence, we give a criterion to decide when $\sG$ is globally given as a complete intersection of $\sF$ with another codimension one distribution.


We must fix some notation in order to state the result mentioned above. When the singular scheme $Z$ of a distribution $\sF$ is 1-dimensional, we let $\mathscr{U}$ denote the maximal 0-dimensional subsheaf of the structure sheaf $\OZ$; the quotient sheaf $\OZ/\mathscr{U}$ is the structure sheaf of a subscheme $C\subset X$ of pure dimension 1. This scheme will be denoted by $\sing_1(\sF)$, and contains all non isolated singularities of $\sF$; the support of the sheaf $\mathscr{U}$ is denoted $\sing_0(\sF)$, and contains all isolated singularities of $\sF$.

\begin{mthm}\label{mthm-sub}
Let $\sF$ be a generic codimension one distribution on a threefold $X$, and  consider a section $\sigma\in H^0(T_\sF\otimes\mathscr{L}^{\vee})$ whose zero locus is a curve $Y$ in $X$, with $\mathscr{L}\in\Pic(X)$. If $\sG$ is the sub-foliation of $\sF$ induced by $\sigma$, then $Y=\sing_1(\sG)$ and $\sing_0(\sG)=\emptyset$. If, in addition, $H^1(\mathscr{L}\otimes\det\Omega^1_X\otimes\det(T_\sF)^2)=0$, then $\sG$  is  given by the intersection of $\sF$ with another codimension one distribution $\sH$ satisfying $\det(N_{\sH})= \det(T_\sF/T_\sG)$. 
\end{mthm}

Tangent sheaves of generic codimension one distributions on $X$ can alternatively be described as quotients of $\Omega^1_X$. It turns out that this is an important class of rank 2 sheaves, providing examples of sheaves with interesting properties in various contexts.


We provide two applications of our ideas to the construction of interesting rank 2 reflexive sheaves on threefolds. We first focus on the case $X=\p3$, establishing the following result. Recall that the degree of a codimension one distribution on $\p3$ is the integer $d:=2-c_1(T_\sF)$.

\begin{mthm}\label{main2}
For each $d\ge0$, $d\ne2$, the moduli space of stable rank 2 reflexive sheaves on $\p3$ with Chern classes
$$ (c_1,c_2,c_3)=(2-d,d^2+2,d^3+2d^2+2d) $$
contains a nonsingular, rational, irreducible component of dimension $(d+1)(d+3)(d+4)/2-1$ whose generic point is the tangent sheaf of a generic distribution of degree $d$ on $\p3$.  
\end{mthm}

The case $d=0$ is well known, since the tangent sheaf of a generic distribution of degree 0 on $\p3$ is precisely $N(1)$, where $N$ is a null correlation bundle, and these are parametrized by an open subset of $\p5$. The case $d=1$ was considered by  (Chang 1984, Theorem 3.14), who showed that the component described in Theorem \ref{main2} is the whole moduli space of stable rank 2 reflexive sheaves with Chern classes $(c_1,c_2,c_3)=(-1,3,5)$; see also (Calvo-Andrade et al. 2018, Theorem 8.1).

\bigskip

For general threefolds, the existence of $\mu$-stable reflexive sheaves with prescribed Chern classes is an open problem with particular interest to String Theory when $X$ is a Calabi--Yau threefold; see Section \ref{cy3f} for more details. In this context, we prove the following  existence and uniqueness result for rank two reflexive sheaves; set
$$ \gamma_X := \min \{ t\in\Z ~|~ \Omega^1_X(t) \textrm{ is globally generated} \}. $$

\begin{mthm}\label{existencia}
Let $X$ be a  smooth projective $3$-fold with rank one Picard group and $c_X<3\rho_X$. Then for every integer $r \geq \gamma_X$, there exists a rank two reflexive sheaf $E$ satisfying:
\begin{itemize}
\item $c_1(E)=c_1( TX)-rc_1(\mathcal{O}_X(1))$,
\item $c_2(E)=c_2( TX)-rc_1(\mathcal{O}_X(1))\cdot c_1( TX)+r^2c_1(\mathcal{O}_X(1))^2$,
\item $c_3(E)=-c_3( TX(-r) ).$
\end{itemize}
Moreover, if $TX$ is simple, then $E$ is uniquely determined by its singular  scheme $\sing(E)$ when $r$ is sufficiently large.
\end{mthm}

Observe that in the Calabi--Yau case we will have
\begin{itemize}
\item $c_1(E)=-rc_1(\mathcal{O}_X(1))$,
\item $c_2(E)=c_2(TX) +r^2c_1(\mathcal{O}_X(1))^2$,
\item $c_3(E)=-c_3(TX)+rc_2(TX)\cdot  c_1(\mathcal{O}_X(1))+r^3 c_1(\mathcal{O}_X(1))^3.$
\end{itemize}

\bigskip

Finally, since foliations always have non isolated singularities, it is therefore also important to study distributions with non isolated singularities. One relevant problem concerning non isolated singularities was formulated by Cerveau in  (Cerveau 2013): if $\sF$ is a codimension one foliation on $\P^3$, then is $\sing_1(\sF)$ connected? We showed in   (Calvo-Andrade et al. 2018) that this question has a negative answer for non-integrable distributions, giving an explicit example of a codimension one distribution on $\p3$ with locally free tangent sheaf whose singular scheme is not connected, see (Calvo-Andrade et al. 2018,  Theorem 9.5 item 2 (b)), thus showing that the connectedness of $\sing_1(\sF)$ must somehow be tied with the integrability of $\sF$.

Our last result shows how to count the number of connected components of $\sing_1(\sF)$ in terms of topological invariants of the tangent sheaf, under certain mild conditions.

\begin{mthm}\label{mthm-conn}
Let $\sF$ be a codimension one distribution on a smooth projective threefold $X$ satisfying $h^1(\OX)=0$. If $h^1(TX\otimes L_\sF^\vee)=h^2(TX\otimes L_\sF^\vee)=0$ and $\sing_1(\sF)$ is reduced, then $\sing_1(\sF)$ has
$$ h^2(T_\sF\otimes L_\sF^\vee) - c_3(T_\sF) + 1$$
connected components. In particular, $\sing_1(\sF)$ is connected if and only if $h^2(T_\sF\otimes L_\sF^\vee) = c_3(T_\sF)$.
\end{mthm}


The results listed above are subsequently proved in the five sections that follow.


\section{Generic distributions on threefolds}

We begin by setting up the notation and nomenclature to be used in the rest of the paper. This section is then completed with the proof of Theorem \ref{StGDist}.

We consider in this paper only codimension 1 distributions on a smooth projective variety $X$ of dimension 3, so that the exact sequence in display \ref{eq:Dist} simplifies to
\begin{equation}\label{dist 3fold}
\mathscr{F}:\  0 \to T_\sF \to TX \to I_{Z/X}\otimes L_\sF \to 0,
\end{equation}
where $L_\sF=\det(TX)\otimes  K_\sF$,   such that  $K_\sF=\det(T_\sF)^\vee$ is the \emph{canonical sheaf} of $\sF$, and $Z$ is the singular scheme of $\sF$. A \emph{foliation} is a distribution $\sF$ satisfying the Frobenius condition $[ T_\sF,  T_\sF]\subset  T_\sF$. Furthermore, we say that a distribution $\sF$ is \emph{generic} if either $Z=\emptyset$ or $\dim Z=0$. 

In addition, we will also assume that $\Pic(X)=\Z$; let $\mathcal{O}_X(1)$ denote its ample generator. We can then assume that $\det(T_\sF)^*=\mathcal{O}_X(-f)$, where $f:=c_1(T_\sF)\in\Z$. We will abuse notation by assuming that the first Chern class of a sheaf is simply given by an integer number, indicating the appropriate multiple of $\mathcal{O}_X(1)$. 

The slope of a torsion free sheaf $E$ on $X$ is given by
$$ \mu_L(E) := \frac{c_1(E)}{\rk(E)} . $$
Recall that $E$ is said to be $\mu$-(semi)stable if every proper nontrivial subsheaf $F\subset E$ satisfies $\mu_L(F)<(\le)~\mu_L(E)$. When $E$ is reflexive, then $E$ is $\mu$-(semi)stable if and only if $h^0(E(p))=0$ for every $p<(\le)~c_1(E)/2$.

Note that if $TX$ is $\mu$-(semi)stable, then $c_X<(\le)~3\rho_X$: since $H^0(\Omega^1_X(\rho_X))\ne0$, there is a nontrivial monomorphism $\OX(-\rho_X)\hookrightarrow\Omega^1_X$, and $\mu$-(semi)stability implies that $-\rho_X<-c_X/3$.

In addition, if $TX$ is $\mu$-(semi)stable, then every subsheaf $F\hookrightarrow TX$ satisfies $c_1(F)<(\le)~2\rho_X$, since $c_1(F)/2<(\le)~c_X/3(\le)~\rho_X$. In other words, the inequalities in the hypotheses of Theorem \ref{StGDist} and Theorem \ref{existencia} can be replaced by the $\mu$-(semi)stability of the tangent bundle of $X$.

\bigskip
\noindent{\bf Proof of Theorem \ref{StGDist}.}
Under our hypotheses, we rewrite the exact sequence in display \eqref{dist 3fold} in the following manner:
\begin{equation}\label{S1}
0 \to T_\sF \longrightarrow TX  \longrightarrow  I_Z(\kappa) \to 0, 
\end{equation}
where $\kappa=c_1(TX)- c_1(T_\sF)$. Set $f=c_1(T_\sF)$. 
Since either  $Z=\emptyset$ or $\dim(Z)=0$, then $$\mathcal{E}xt^1( I_Z(d), \O)=0.$$  Therefore, dualizing the sequence in display \eqref{S1} we obtain 
\begin{equation}\label{S2}
 0 \to \O(-\kappa)   \longrightarrow  \Omega_X^1  \longrightarrow   (T_\sF)^{\vee} \to 0.
 \end{equation}
Recall that $(T_\sF)^{\vee} \simeq T_\sF(-f)$ since $T_\sF$ is a rank $2$ reflexive sheaf. Now, twisting the sequence in display \eqref{S2} by $\O(p+f)$ we get
$$ 0 \to \O(p+f-\kappa)   \longrightarrow  \Omega_X^1(p+f)  \longrightarrow   T_\sF(p)\to 0. $$
Since   $H^1(\mathcal{O}_X(t))=0$ for all $t\in \mathbb{Z}$, the induced map in cohomology 
$$ H^0(X, \Omega_X^1(p+f)) \to H^0(X,T_\sF(p)) $$
is surjective for every $p\in\Z$. Thus, if  $h^0(X,T_\sF(p))  \neq 0$, then $p+f\geq \rho_X $, i.e, $p\geq \rho_X-f$.
It follows that 
$$ p\geq \frac{f}{2}-f=- \frac{f}{2}, $$
thus $T_\sF$ is $\mu$-semistable. If $c_1(T_\sF)=f<2\rho_X$, then 
$$ p > -\frac{f}{2}, $$
Thus $T_\sF$ is $\mu$-stable.\qed

\begin{Obs}
Theorem \ref{StGDist} generalizes (Calvo-Andrade et al. 2018, Theorem 6.1), which covered the case $X=\p3$.
\end{Obs}


\section{Sub-foliations of codimension one distributions}

A sub-foliation of a codimension one distribution $\sF$ on $X$ is a codimension two distribution $\sG$ whose tangent sheaf $T_{\sG}$ is a subsheaf of $T_\sF$. Since $\dim X=3$, note that $T_{\sG}$ must be a line bundle, hence $\sG$ is given by a twisted vector field $\phi\in H^0(TX\otimes T_\sG^\vee)$ and is automatically integrable. Moreover, the morphism $T_\sG\hookrightarrow T_\sF$ can be regarded as a section $\sigma\in H^0(T_\sF\otimes T_\sG^\vee)$. Conversely, any section $\sigma\in H^0(T_\sF\otimes\mathscr{L}^\vee)$ for $\mathscr{L}\in\Pic(X)$ vanishing in codimension 2 induces a subfoliation $\sG$ of $\sF$ such that $T_\sG=\mathscr{L}$.

The quotient $N_{\sF/\sG}:=T_\sF/T_\sG$ is called the \emph{relative normal sheaf}; it satisfies the following short exact sequence
$$ 0 \to N_{\sF/\sG} \to N_{\sG} \to N_{\sF} \to 0; $$
in particular, $N_{\sF/\sG}$ is a torsion free sheaf of rank 1. In the case at hand, $N_{\sF/\sG}$ is the ideal sheaf of the (possibly empty) 1-dimensional scheme $Y:=(\sigma)_0$, the vanishing locus of the section $\sigma$, twisted by the  line bundle  $L_{\sF/\sG}:=T_\sG^\vee\otimes\det(T_\sF)$. Rewritting the previous exact sequence, we observe that the normal sheaf of $\sG$ can be described as an extension of twisted ideal sheaves:
\begin{equation}\label{ideals}
0 \to I_Y\otimes L_{\sF/\sG} \to N_\sG \to I_Z\otimes L_\sF \to 0
\end{equation}

The goal of this section is to describe the relations between $Z:=\sing(\sF)$, $W:=\sing(\sG)$ and the curve $Y$. 

The dualization of the exact sequence in display \eqref{ideals} yields a (possibly trivial) morphism
$$\xi:L_{\sF/\sG}^\vee \to \inext^1(I_Z\otimes L_\sF,\OX)\simeq\omega_C\otimes\det(T_\sF),$$
which can be regarded as a section $\xi\in H^0(\omega_C\otimes\det(T_\sF)\otimes L_{\sF/\sG})$. The kernel of $\xi$ can be written as $I_D\otimes L_{\sF/\sG}^\vee$, the twisted ideal sheaf of a curve $D\subset X$, so that $\img\xi=\mathcal{O}_D\otimes L_{\sF/\sG}^\vee$; define also $V:=\coker\xi$. We therefore obtain the following 4 short exact sequences:
\begin{gather}
0 \to L_\sF^\vee \to N_\sG^\vee \to I_D\otimes L_{\sF/\sG}^\vee \to 0 \label{i} \\
0\to \mathcal{O}_D\otimes L_{\sF/\sG}^\vee \to \omega_C\otimes\det(T_\sF) \to V \to 0 \label{ii}\\
0\to V \to \mathcal{O}_W\otimes T_{\sG}^\vee \to S \to 0 \label{iii}\\
0\to S \to \omega_Y\otimes\det(TX)\otimes L_{\sF/\sG}^\vee \to \inext^2(I_Z,\OX) \to 0. \label{iv}
\end{gather}
Here, we also used the following isomorphisms:
$$ \inext^1(I_Y\otimes L_{\sF/\sG},\OX) \simeq \omega_Y\otimes\det(TX)\otimes L_{\sF/\sG}^\vee
~~{\rm and}~~ 
\inext^1(N_\sG,\OX) \simeq \mathcal{O}_W\otimes T_\sG^\vee. $$
With these sequences in mind, we establish the following result.

\begin{Prop}\label{prop-sub}
Let $\sF$ be a codimension one distribution of on a threefold $X$, and let $\sigma\in H^0(T_\sF\otimes\mathscr{L}^\vee)$ be a section such that $Y:=(\sigma)_0$ is a (possibly empty) curve in $X$, with $\mathscr{L}\in\Pic(X)$. If $\sG$ is the sub-foliation of $\sF$ induced by $\sigma$, then $Y\subseteq\sing_1(\sG)$, with equality holding if and only if $\dim\coker\xi=0$. In particular, if $\sing_1(\sF)$ is nonempty, irreducible and reduced, then either $Y=\sing_1(\sG)$ or $\sing(\sG)=Y\cup\sing_1(\sF)$.
\end{Prop}
\begin{proof}
Since $\dim\inext^2(I_Z,\OX)=0$, the sequence in display \eqref{iv} implies that the support of the sheaf $S$ is precisely $Y$, hence $Y\subset W=\sing(\sG)$ by the sequence in display \eqref{iv}. This sequence also tells us that $Y$ coincides with the 1-dimensional component of $W$ if and only if $\dim V=0$.

If $\sing_1(\sF)$ is is nonempty and irreducible, then the sequence in display \eqref{ii} implies that either $D=C$ (if $\xi\ne0$) or $D=\emptyset$ (if $\xi=0$). In the first case, $\dim V=0$, and we conclude that $Y=\sing_1(\sG)$; in the second case, we have that $V\simeq\omega_C\otimes\det(T_\sF)$, and the sequence in display \eqref{iii} tells us that $W=C\cup Y$.
\end{proof}

\begin{Obs}
Now assume that $C$ is nonempty, connected and reduced. If $\omega_C\otimes\det(T_\sF)\otimes L_{\sF/\sG}\simeq\mathcal{O}_C$, then $\xi\in H^0(\mathcal{O}_C)$ is either zero or nowhere vanishing. In the first case, we again have that $D=\emptyset$, thus $\sing_1(\sG)=Y\cup\sing_1(\sF)$. If the second possibility occurs, then $V=0$ and again we conclude that $Y=\sing_1(\sG)$. \end{Obs}




Theorem \ref{mthm-sub} can now be obtained by taking $\sF$ to be a generic codimension one distribution, so that $C=\emptyset$. It follows that the sequence in display \eqref{i} simplifies to
\begin{equation} \label{i'}
0 \to L_\sF^\vee \to N_\sG^\vee \to L_{\sF/\sG}^\vee \to 0;
\end{equation}
In particular, the conormal sheaf $N_\sG^\vee$ is locally free, so $\sing_0(\sG)=\emptyset$. Morever, we have that $V=0$, thus $S\simeq\mathcal{O}_W\otimes\mathscr{L}^\vee$. Since $\dim\inext^2(I_Z,\OX)=0$, the sequence in display \eqref{iv} implies $W=Y$.

Note that if $H^1(L_{\sF/\sG}\otimes L_\sF^\vee)=0$, then the sequence above must be a trivial extension, thus
$$N_\sG^\vee=L_\sF^\vee \oplus L_{\sF/\sG}^\vee.$$
In other words, the subfoliation is given by the intersection of the    distribution $\sF$ with other distribution $\sH$ whose $\det(N_{\sH})=L_{\sF/\sG}$. 

We have therefore completed the proof of Theorem \ref{mthm-sub}.

\begin{Obs}
Proposition \ref{prop-sub} above generalizes (Calvo-Andrade et al. 2018, Lemma 3.6), in which we implicitly used the irreducibility and reducedness of $\sing_1(\sF)$.  
\end{Obs}


\section{Moduli spaces of rank 2 reflexive sheaves on $\p3$}

In this section we focus on the case $X=\p3$. Given a generic distribution $\sF$, we set $d:=2-c_1(T_\sF)\ge0$, which is called the \emph{degree} of $\sF$.

Theorem \ref{StGDist} implies that $T_\sF$ is a stable rank 2 reflexive sheaf for every $d\ge0$; its second and third Chern classes in terms of the degree $d$ are given by
\begin{align*}
c_2(T_\sF)&= d^2+2, ~~{\rm and} \\
c_3(T_\sF)&= d^3+2d^2+2d=h^0(\mathcal{O}_Z),
\end{align*}
see  (Calvo-Andrade et   al. 2018, equations (18) and (19)).

Our goal is to show that the moduli space space of stable rank 2 reflexive sheaves on $\p3$ with Chern classes given by
$$ (c_1,c_2,c_3)=(2-d,d^2+2,d^3+2d^2+2d), $$
that is, equal to those of the tangent sheaf of a generic distribution, contains a nonsingular irreducible component of dimension $h^0(\Omega_{\p3}^1(d+2))-1$ whose points are sheaves $F$ given by an exact sequence of the form
\begin{equation}\label{sqc-f}
0 \to \op3(-2d) \stackrel{\sigma}{\rightarrow} \Omega^1_{\p3}(2-d) \to F \to 0.
\end{equation}
We will denote such irreducible component simply by $\mathcal{R}(d)$. Note that the exact sequence in display \eqref{sqc-f} is exactly the same as the one in display \eqref{S2} rewritten in terms of the degree $d$; dualizing \eqref{sqc-f} yields precisely the sequence in display \eqref{dist 3fold} up to a twist by $\op3(2-d)$, with $Z$ being the singular locus of the sheaf $F$, that is $\mathcal{O}_Z=\inext^1(F,\op3)$.

The strategy for the proof of Theorem \ref{main2} is as follows. The family of sheaves of given by the exact sequence in display \eqref{sqc-f} has dimension $h^0(\Omega_{\p3}^1(d+2))-1$, since each $F$ is defined by a section $\sigma\in H^0(\Omega_{\p3}^1(d+2))$ up to a scalar multiple, so we must argue that 
$$ \dim\mathcal{R}(d) = \dim\Ext^1(F,F) = h^0(\Omega_{\p3}^1(d+2))-1 . $$
Invoking (Hartshorne 1980, Proposition 3.4), we have that
$$ \dim\Ext^1(F,F) - \dim\Ext^2(F,F) = 8c_2(F)-2c_1(F)^2-3=6d^2+8d+5, $$
since $F$ is stable, see (Hartshorne 1980, Remark 3.4.1). We must therefore compute the dimension of $\Ext^2(F,F)$, showing that
\begin{equation} \label{key}
\dim\Ext^2(F,F) = h^0(\Omega_{\p3}^1(d+2))-6d^2-8d-6=\frac{1}{2}d(d-1)(d-3).
\end{equation}
We will show that the previous equality holds whenever $d\ne2$.

The first step is the following lemma, whose proof of a straight forward calculation using the exact sequence in cohomology derived from the exact sequence \eqref{sqc-f}.

\begin{Lemma}\label{cohomology}
If a sheaf $F$ satisfies the exact sequence in display \eqref{sqc-f}, then:
\begin{enumerate}
\item $h^0(F(p))=0$ for $p\le d-1$;
\item $h^1(F(p))=0$ for $p\ne d-2$, and $h^1(F(d-2))=1$;
\item $h^2(F(p))=h^0(\op3(2d-p-4))$ for $p\ge d-4$; in particular, $h^2(F(p))=0$ for $p\ge 2d-3$;
\item $h^3(F(p))=0$ for $p\ge d-4$.
\end{enumerate}
\end{Lemma}

Applying the functor $\Hom(\cdot,F)$ to the exact sequence in display \eqref{sqc-f}, we obtain
\begin{equation}\label{isom1}
\Ext^2(F,F) \simeq \Ext^2(\Omega^1_{\p3}(2-d),F) = H^2(F\otimes \tp3(d-2)),
\end{equation}
since $H^1(F(2d))=H^2(F(2d))=0$ by Lemma \ref{cohomology}.

Next, we twist the exact sequence in display \eqref{sqc-f} by $\tp3 (d-2)$ and pass to cohomology, obtaining the isomorphism
\begin{equation}\label{isom2}
H^1(F\otimes \tp3(d-2)) \simeq H^2(\tp3(-d-2))
\end{equation}
since $H^1(\Omega^1_{\p3}\otimes \tp3)=H^2(\Omega^1_{\p3}\otimes \tp3)=0$. It follows that
$h^1(F\otimes \tp3(d-2))=0$ when $d\ne2$, and $h^1(F\otimes \tp3(-4))=1$.

Finally, we twist the Euler sequence by $F(d-2)$, obtaining the exact sequence in cohomology
$$ 0 \to H^1(F\otimes \tp3(d-2)) \to H^2(F(d-2)) \to H^2(F(d-1)^{\oplus4}) \to H^2(F\otimes \tp3(d-2)) \to 0, $$
since $h^1(F(d-1))=h^3(F(d-2))=0$ by Lemma \ref{cohomology}. Using item (3) of Lemma \ref{cohomology}, and the isomorphisms
\eqref{isom1} and \eqref{isom2}, we obtain the formula
$$ \dim\Ext^2(F,F) = \left\{ \begin{array}{l}
0, ~~ {\rm for} ~~ d\le2  \\ 4\cdot h^0(\op3(d-3)) - h^0(\op3(d-2)), ~~ {\rm for} ~~ d\ge3.
\end{array} \right. $$
A simple calculation shows that 
$$ 4\cdot h^0(\op3(d-3)) - h^0(\op3(d-2)) = \frac{1}{2}d(d-1)(d-3), $$
thus establishing the equality in display \eqref{key} when $d\ne2$.

The rationality of $\mathcal{R}(d)$ is simply the fact that it contains an open subset which is isomorphic to an open subset of $\mathbb{P}H^0(\Omega^1_{\p3}(d+2))$. We have therefore completed the proof of Theorem \ref{main2}.

\bigskip

We conclude this section pointing out two interesting facts regarding the tangent sheaves of generic codimension one distributions on $\p3$. The first one is described in the following result.

\begin{Lemma}
Let $F$ be the tangent sheaf of a generic codimension one distribution of degree $d\ge1$ on $\p3$. Then $F(d)$ is globally generated, and satisfies the exact sequence
$$ 0\to \tp3(-2)\oplus\op3(-d) \to \op3^{\oplus6} \to F(d) \to 0. $$
\end{Lemma}

Again, the case $d=0$ is well-known; as we mentioned in the Introduction, $F=N(1)$ for a null correlation bundle $N$, and it satisfies the well-known exact sequence 
$$ 0\to \tp3(-2) \to \op3^{\oplus5} \to F \to 0. $$
In addition, the case $d=1$ is also considered by Chang, see the proof of Theorem 3.14 in \cite{Chang}.

\begin{proof}
Our starting point is the exact sequence
$$ 0\to \tp3(-2) \to \op3^{\oplus6}\to \Omega^1_{\p3}(2) \to 0. $$
Composing the epimorphisms $\op3^{\oplus6}\twoheadrightarrow \Omega^1_{\p3}(2)$ and $ \Omega^1_{\p3}(2)\twoheadrightarrow F(d)$, we obtain the epimorphism $\varphi:\op3^{\oplus6}\twoheadrightarrow F(d)$, showing that $F(d)$ is globally generated. Notice that $h^0(F(d))=h^0(\Omega^1_{\p3}(2))=6$.

Moreover, a diagram chase shows that $\ker\varphi$ is an extension of $\op3(-d)$ by $\tp3(-2)$; however, $$\Ext^1(\op3(-d),\tp3(-2))=H^1(\tp3(d-2))=0$$ for every $d\ge0$, thus $\ker\varphi=\tp3(-2)\oplus\op3(-d)$, as desired.
\end{proof}

\begin{Obs}\rm
When $d=2$, we can still conclude  that the sheaves $F$ given by
$$ 0\to \op3(-4) \to \Omega^1_{\p3} \to F \to 0 $$
are smooth points of the moduli space of stable rank 2 reflexive sheaves with Chern classes $(c_1,c_2,c_3)=(0,6,20)$ within an irreducible component of dimension 45, since $\Ext^2(F,F)=0$. However, these sheaves only form a family of dimension 44 within this irreducible component.\qed
\end{Obs}


The next result show us that for each generic distributions there exist a family of   smooth connected curves passing through  all its singular points. 
\begin{Prop}
For each generic codimension one distribution $\sF$ of degree $d\ge 1$ on $\p3$, there is a family of  smooth connected curves  of degree $d^2+2d+2$  and arithmetic genus $(d-1)(d^2+2d+2)+1$ passing  through  the $d\cdot(d^2+2d+2)$ singular points of $\sF$.
\end{Prop} 
\begin{proof}
Let $F$  be  the tangent sheaf of a generic codimension one distribution degree $d\ge1$ on $\p3$. Since $h^0(F(d-1))=0$, the zero locus of an arbitrary section $\sigma\in H^0(F(d))$ is a curve $C$ of degree $c_2(F(d))=d^2+2d+2$ containing the singular points of $F$, which coincides with the singular points of the distributions. On the one hand, we have that
\begin{equation}\label{c31}
c_3(F(d))= c_3(F)=d(d^2+2d+2)
\end{equation}
On the other hand, from (Hartshorne 1980, Theorem 4.1)  we obtain 
\begin{equation}\label{c32}
c_3(F(d))= 2p_a(C)-2+ c_2(F(d))(4-c_1(F(d)))=2p_a(C)-2 +(d^2+2d+2)(2-d) . 
\end{equation}
By using (\ref{c31}) and  (\ref{c32}) we conclude that
$$
p_a(C)=(d-1)(d^2+2d+2)+1.
$$

Since $d\ge1$, it follows from the exact sequence 
$$
0 \to \op3(-d)  \rightarrow \Omega^1_{\p3}(2) \to F(d) \to 0
$$
that  $H^0(F(d))\simeq H^0(\Omega^1_{\p3}(2))$. 
Thus a generic section $\sigma\in H^0(F(d))$ lifts to a morphism $$\op3\to\Omega^1_{\p3}(2)$$  whose cokernel is $N(1)$, a twisted null correlation bundle. Furthermore, we also obtain the exact sequence
$$ 0\to \op3(-d) \to N(1) \to I_C(d+2)\to 0, $$
thus $C$ is also the zero locus of a generic section of $N(d+1)$. Since the latter is globally generated and $H^1(N(1+d)^\vee)=H^1(N(-d-1))=0$ for $d\ge1$, we can invoke (Hartshorne 1978, Proposition 1.4)  to conclude that $C$ is smooth and connected, as desired. Moreover, observe that by construction  these curves corresponds to a family  of dimension equal to $h^0(F(d))-1= h^0(\Omega^1_{\p3}(2))-1=5.$ 
\end{proof}


\section{Existence of Rank 2 reflexive sheaves  threefolds} \label{cy3f}

This section is dedicated to the study of the existence of rank two reflexive sheaves on threefolds with prescribed Chern classes, closing with the proof of Theorem \ref{existencia}.

To be precise, fix a smooth projective variety $X$ of dimension 3, and consider the following set:
\begin{equation}\label{tilde-s}
\widetilde{S_r}(X):=\left\{ (C,D,S)\in\bigoplus_{i=1}^3 H^{2i}(X,\Z) ~  \right. \left| 
\begin{array}{l}
\textrm{there is a } \mu\textrm{-stable reflexive sheaf of rank } r  \\
\textrm{ with } (c_1(F),c_2(F),c_3(F))=(C,D,S)
\end{array} \right \}.
\end{equation}
The Picard group of $X$ acts on $\widetilde{S_r}(X)$ in the following way:
\begin{align*}
\Pic(X) \times \widetilde{S_r}(X) & \to \widetilde{S_r}(X) \\
L\cdot(R,D,S) & \mapsto (R+c_1(L),D+c_1(L)\cdot R + c_1(L)^2,S),
\end{align*}
and we defined the quotient set $S_r(X):=\widetilde{S_r}(X)/\Pic(X)$, which we call the \emph{rank $r$ stable spectrum of $X$}.

Determining the stable spectrum of a given threefold is not an easy task, and only the case $X=\p3$ with $r=2$ is completely understood, see (Mir\'o-Roig  1985). Some progress was made by the third named author when $X$ is a hypersurface in $\p4$, see  (Jardim 2007).

More recently, this issue became of considerable interest within the context of String Theory for the case when $X$ is a Calabi--Yau manifold, see for instance the articles (Andreas $\&$ Curio 2011, Gao et al 2015, Wu $\&$ Yau 2014). For instance, it follows from (Wu $\&$ Yau 2014, Theorem 2)
 that $(0,mH^2,2mH^3)\in S_k(X)$ for $m\geq k$, for some very  ample divisor $H$. In addition, a conjecture by Douglas, Reinbacher and Yau provides an upper bound for the third Chern class of stable bundle on simply connected Calabi--Yau threefolds, see  (Douglas 2006, Conjecture 1.1).   
Letting $X$ be a smooth projective threefold satisfying the hypotheses of Theorem \ref{StGDist}, our third theorem can be regarded as providing some information on the rank 2 stable spectrum of $X$. To be precise, Theorem \ref{existencia} implies that
$$ \left( c_1( TX)-rH,c_2( TX)-rc_1(TX)\cdot H+r^2H^2,-c_3(TX(-r)) \right) \in \widetilde{S_r}(X) $$
for each $r\ge\gamma_X$, where $H=c_1(\OX(1))$.

\bigskip

\noindent{\bf Proof of Theorem \ref{existencia}.}
For each $r\ge\gamma_X$, we can apply Ottaviani's Bertini type Theorem (Ottaviani 1995, Theorem 2.8) to conclude that the degeneration locus of a generic morphism $\omega\in\Hom(TX,\OX(r))$ is either empty or 0-dimensional. Defining $F:=\ker(\omega)$, we have the exact sequence
$$ 0\to F \to TX \to I_Z(r)\to 0, $$
where $Z$ denotes the degeneration locus of $\omega$. In other words, $F$ is the tangent sheaf of a generic distribution on $X$.

Observe that $c_1(TX)-c_1(F)= r\geq\gamma_X\geq\rho_X$. By hypothesis $-c_1(TX)>-3\rho_X$, thus we have that
$$ -c_1(F) \geq \rho_X -r - c_1(F)= -c_1(TX)+\rho_X  > -2\rho_X.  $$
Theorem \ref{StGDist} then implies that $F$ is stable; its Chern classes are
\begin{itemize}
\item $c_1(F)=c_1(TX)-rc_1(\OX(1))$,
\item $c_2(F)=c_2(TX)-rc_1(\OX(1))\cdot c_1(TX)+r^2c_1(\OX(1))^2$,
\item $c_3(F)=c_3(\Omega_X^1(r))=-c_3(TX(-r))$,
\end{itemize}
see (Cavalcante et al. 2020, Theorem 5.6).

Finally, suppose that $TX$ is simple. We can take $m_0$ such that 
$H^i(X, \Omega^1_X\otimes\wedge^{i+1}TX(-ir) )=0,$
for $i=1,2$. Follows from (Araujo $\&$ Corr\^ea 2013, Theorem 1.1)  that if $$\sing(F')=\{\omega'=0\}\subset \sing(F)= \{\omega=0\},$$ then
there is $\lambda \in \mathbb{C}\simeq H^0(X, End(TX))$ such that $\omega'= \lambda \omega$. In particular, $F=\ker(\omega)=\ker(\omega')=F'$.


\section{Connectedness of non isolated singularities} \label{conn}

We now shift our attention to codimension one distributions on threefolds containing non isolated singularities. As we mentioned at the Introduction, we have the following short exact sequence
\begin{equation}\label{u-oz-oc}
0\to \mathscr{U} \to \OZ \to \OC \to 0,    
\end{equation}
where $\mathscr{U}$ is the maximal 0-dimensional subsheaf of $\OZ$ and $C=\sing_1(\sF)$. Our next result, in which we do not assume that $\Pic(X)=\Z$, describes the number of connected components of $C$.

\begin{Prop}\label{prop-conn}
Let $\sF$ be a codimension one distribution on a smooth projective threefold $X$ satisfying $h^1(\OX)=0$; let $C:=\sing_1(\sF)$. If $h^1(TX\otimes L_\sF^\vee)=h^2(TX\otimes L_\sF^\vee)=0$, then 
$$ h^0(\OC)=h^2(T_\sF\otimes L_\sF^\vee) - c_3(T_\sF) + 1.$$
\end{Prop}

In particular, if $C$ is reduced, then $h^0(\OC)$ is precisely the number of connected components $C$ has, and the statement of Theorem \ref{mthm-conn} follows easily.

\begin{proof}
The sequence in display \eqref{u-oz-oc} implies that $h^0(\OC)=h^0(\OZ)-h^0(\mathscr{U})$; we analyze the terms in the right hand side of this last equality separately.

The standard exact sequence
$$ 0\to \IZ \to \OX \to \OZ \to 0 $$
implies that $h^1(\IZ)=h^0(\OZ)-h^0(\OX)=h^0(\OZ)-1$, since $h^1(\OX)=0$ by hypothesis. Twisting the sequence in display \eqref{dist 3fold} by $LK_\sF^\vee$, we obtain that
$$H^1(\IZ)\simeq H^2(T_\sF\otimes L_\sF^\vee),$$ since we assumed that $h^1(TX\otimes L_\sF^\vee)=h^2(TX\otimes L_\sF^\vee)=0$. It follows that 
$$h^0(\OZ)=h^2(T_\sF\otimes L_\sF^\vee)+1.$$
Dualizing the sequences in display \eqref{dist 3fold} and \eqref{u-oz-oc}, we conclude that
$$ \inext^1(T_\sF,\OX)\simeq\inext^2(\IZ\otimes L_\sF,\OX)\simeq 
\inext^3(\mathscr{U},\OX). $$
Note that
$$ H^0(\inext^3(\mathscr{U},\OX)) \simeq \Ext^3(\mathscr{U},\OX) \simeq H^0(U\otimes\omega_X)^*; $$
the first isomorphism follows from the spectral sequence for local to global ext's (since $\dim\mathscr{U}=0$), while the second isomorphism is given by Serre duality. It follows that  $$ h^0(\mathscr{U})=h^0(\inext^3(\mathscr{U},\OX))=h^0(\inext^1(T_\sF,\OX)). $$
However, the last dimension coincides with $c_3(T_\sF)$ ( Hartshorne 1980, Proposition 2.6)  for the proof of this fact when $X=\p3$, the general case being similar.

Gathering all the terms, we obtain the formula in the statement of the proposition. As for the second claim, just note that if $C$ is reduced, then $h^0(\OC)$ is precisely the number of connected components of $C$.
\end{proof}

Specializing to the case $X=\p3$, we obtain the following claim, which generalizes ( Calvo-Andrade et al. 2018, Theorem 3.8).

\begin{Cor}
Let $\sF$ be a codimension one distribution on $\p3$ of degree $d$, and let $C=\sing_1(\sF)$. Then
$$ h^2(T_\sF(-d-2)) - c_3(T_\sF) \le h^0(\OC) \le h^2(T_\sF(-d-2)) - c_3(T_\sF) + 1. $$
If $d\ne2$, then $h^0(\OC) = h^2(T_\sF(-d-2)) - c_3(T_\sF) + 1$.
In particular, if $C$ is reduced, then $C$ is connected if and only if $h^2(T_\sF(-d-2)) = c_3(T_\sF)$.
\end{Cor}
\begin{proof}
For $d\ne2$, the claim follows from considering $X=\p3$ in Proposition \ref{prop-conn}. 
When $d=2$, we have that $h^2(TX\otimes L_\sF^\vee)=H^2(\tp3(-4))=1$, so one cannot apply Proposition \ref{prop-conn} directly. However, the cohomology sequence associated to the sequence in display \eqref{S1} ($\kappa=4$ in this case) is
$$ 0 \to H^1(\IZ) \to H^2(T_\sF(-4)) \to H^2(\tp3(-4)) \to \cdots $$
thus either $h^1(\IZ)=h^2(T_\sF(-4))$ or $h^1(\IZ)=h^2(T_\sF(-4))-1$. 
\end{proof}

\subsection*{Acknowledgments}
MC is partially supported by CNPq
grant numbers
202374/2018-1, 302075/2015-1, 400821/2016-8,  CAPES and Fapemig; he is grateful to University of Oxford for its hospitality, in special he is thankful to Nigel Hitchin for interesting mathematical conversations.
MJ is partially supported by the CNPq grant number 303332/2014-0 and by the FAPESP Thematic Project number 2018/21391-1.


\end{document}